\documentclass[12pt, reqno]{amsart}
\usepackage{graphicx, epsfig, psfrag}
\usepackage{amsfonts,amsmath,amssymb}
\usepackage{xcolor}
\usepackage{float}
\usepackage{hyperref}
\usepackage{theorem}
\usepackage{mathrsfs}
\usepackage{pdfsync}
\usepackage{subfig}
\usepackage{stmaryrd}
\usepackage[left=2.5cm,right=2.5cm,top=3cm,bottom=3cm]{geometry}

\newcounter{theorem}
\def\openthm#1#2{\refstepcounter{theorem}\bigskip

{\noindent\bf#1~\thetheorem\if#2!{. }\else{ (#2).}\fi}
\it}

\def\thmskip{}
\newenvironment{theorem}[1][!]{\openthm{Theorem}{#1}}{\thmskip}
\newenvironment{lemma}[1][!]{\openthm{Lemma}{#1}}{\thmskip}

\newenvironment{proposition}[1][!]{\openthm{Proposition}{#1}}{\thmskip}

\newcounter{remark}
\def\openrem#1#2{\refstepcounter{remark}\bigskip
{\noindent \it \bfseries#1~\theremark\if#2!{. }\else{ (#2). }\fi}}
\newenvironment{remark}[1][!]{\openrem{Remark}{#1}}{\qed}

\newcounter{algorithm}
\def\openalg#1#2{\refstepcounter{algorithm}\bigskip
{\noindent \it \bfseries#1~\thealgorithm\if#2!{. }\else{ (#2). }\fi}}
\newenvironment{algorithm}[1][!]{\openalg{Algorithm}{#1}}{\qed}

\newcounter{definition}
\def\opendef#1#2{\refstepcounter{definition}\bigskip
{\noindent \bf#1~\thedefinition\if#2!{. }\else{ (#2). }\fi}\it}

\def\R{\mathbb{R}}
\def\N{\mathbb{N}}
\def\Z{\mathbb{Z}}

\def\dx{\,{\rm d}x}

\def\<{\langle}
\def\>{\rangle}

\def\argmin{{\rm argmin}}

\def\eps{\varepsilon}
\def\Bs{\mathcal{B}}
\def\Es{\mathcal{E}}
\def\Etot{E}
\def\Us{\mathcal{U}}

\def\Ts{\mathcal{T}}
\def\Ps{\mathcal{P}}

\def\Ys{\mathcal{Y}}
\def\Ks{\mathcal{K}}

\def\a{{\rm a}}
\def\qc{{\rm qc}}

\def\c{{\rm c}}

\renewcommand{\cases}[1]{\left\{ \begin{array}{rl} #1 \end{array} \right.}

\newcommand{\smfrac}[2]{{\textstyle \frac{#1}{#2}}}
\numberwithin{equation}{section}

\begin{document}

\title[A Posteriori Error Control for a QC Approximation]{A Posteriori
  Error Control for a Quasicontinuum Approximation of a Periodic
  Chain}

\author{Christoph Ortner}
\address{Christoph Ortner, Mathematics Institute, Zeemen Building,
  Universty of Warwick, Coventry CV4 7AL, UK}
\email{c.ortner@warwick.ac.uk}

\author{Hao Wang}
\address{Hao Wang, Oxford University Mathematical Institute,
  24-29 St Giles', Oxford, OX1 3LB, UK}
\email{wangh@maths.ox.ac.uk}

\date{\today}

\thanks{CO was supported by an EPSRC Grant ``Analysis of
  Atomistic-to-Continuum Coupling Methods''. HW was supported by China Scholarship Council and University of Oxford's 
``Chinese Ministry of Education-University of Oxford Scholarship'' and Sichuan University's National high Level University Construction Programme}

\begin{abstract}
  We consider a 1D periodic atomistic model, for which we formulate
  and analyze an adaptive variant of a quasicontinuum method. We
  establish a posteriori error estimates for the energy norm and for
  the energy, based on a posteriori residual and stability
  estimates. We formulate adaptive mesh refinement algorithms based on
  these error estimators. Our numerical experiments indicate optimal
  convergence rates of these algorithms.
\end{abstract}

\maketitle

\section{Introduction}
Quasicontinuum (QC) methods, or more generally, atomistic-to-continuum
coupling (a/c) methods, are a class of multiscale methods for coupling
an atomistic model of a solid with a continuum model. These methods
have been widely employed in atomistic simulations, where a fully
atomistic model would result in prohibitive computational cost but is
required in certain regions of interest (typically neighbourhoods of
crystal defects) to achieve the desired degree of accuracy. A
continuum model is applied to reduce the cost of the computation of
the elastic far-fields. We refer to the recent review articles
\cite{LuOr:Encyclopedia,Miller2009a} for introductions to QC methods.

In this paper we present an a posteriori error analysis of a simplified
variant of the energy-based QC approximation proposed in
\cite{Shapeev2010a}.



Considerable effort has been devoted to the a priori error analysis of
QC methods \cite{Lin:2003a, Lin:2007a, Ortner:2008a, Dobson:2009a,
  Dobson:2009b, Ortner:arXiv:0911.0671, OrtShap2010b,
  VanKotenLuskin2010a, Ortner2011a, Ming2009, LiLuskin2011a}. A
posteriori error control has received comparatively little attention:
Arndt and Luskin use a goal-oriented approach for the a posteriori
error estimates for the QC approximation of a Frenkel-Kontorova model
\cite{Arndt2008a, Arndt2008b, Arndt2007a}. The estimates are used to
optimize the choice of the atomistic region and the finite element
mesh in the continuum region. Prudhomme et al. \cite{Serge2007a} also
use the goal-oriented approach to provide a posteriori error control
for the original energy-based QC approximation \cite{Ortiz:1995a},
which is inconsistent. Both of these approaches require the use of the
solutions of dual problems. In \cite{Ortner:2008a} a posteriori error
bounds for an energy norm are derived, using a similar approach as in
the present work. However, the QC method analyzed in
\cite{Ortner:2008a} does not contain an approximation of the stored
energy and its 2D/3D version is therefore not a computationally
efficient method.

In the present work, we use the residual-based approach \cite[Chapter
2]{VerfurthAPosteriori}, which is well established in the finite
element approximation of partial differential equations, to derive a
posteriori error bounds for the energy-norm and for the energy
itself. What distinguishes our setting from the classical one, are two
particular features: 1. the model approximation (``variational
crime'') is fundamentally different than quadrature approximations;
2. it is possible to adapt the mesh to a degree that the residual
vanishes exactly in certain regions of interest (the atomistic
region). Our work extends \cite{Ortner:arXiv:0911.0671}, which
considers a simplified setting.

We do not tackle the question of model and mesh adaptivity for defect
nucleation, but focus only on automatically choosing the {\em size} of
the atomistic region and the finite element mesh in the continuum
region. In many applications, a defect is inserted into the crystal
before the computation, or it is known a priori where a defect will
nucleate. An extension of our work to include defect nucleation would
be valuable, but cannot be pursued in the simplified 1D setting we are
considering here.

\subsection{Outline}
In Section \ref{Sec:A_and_Q_Model} we introduce a 1D atomistic model
problem, which mimics the behaviour of crystal defects in
2D/3D. Moreover, we review the construction of a QC method to
efficiently approximate its solutions, and introduce the notation that
will be used throughout the paper.

In Section \ref{Sec:Residual_Analysis}, we derive a residual estimate
for the QC method in a discrete negative Sobolev norm.  In Section
\ref{Sec:Stability}, we present an a posteriori stability result.  In
Section \ref{Sec:A_Posteriori_Error}, we combine the residual estimate
and the stability result to give a posteriori error estimates for the
deformation gradient and for the energy.

In Section \ref{Numerics}, we describe three mesh refinement
algorithms based on our a posteriori error analysis and on previous a
priori error estimates, and present a numerical example to illustrate
the performance of these algorithms.

\section{Model Problem and QC Approximation}
\label{Sec:A_and_Q_Model}
\subsection{Atomistic Model}
\label{subsec:AModel}
%
Following previous works \cite{Dobson:2009a, Ortner:arXiv:0911.0671,
  Wang:2011a} we formulate a model problem in a discrete periodic
domain containing $2N$ atoms, where $N \in \mathbb{N}$.  Let $F > 0$
denote a macroscopic stretch and $\eps = 1/(2N)$ the lattice spacing,
both of which we fix throughout. We define the displacement and
deformation spaces, respectively, by
\begin{align*}
  \Us^\eps :=~& \big\{ u \in \R^\Z: u_{\ell+2N} = u_\ell, \text{ and }
  u_0 = 0 \big\}, \quad \text{and} \\
 %
  \Ys^\eps :=~&  \big\{ y \in \R^\Z: y_{\ell+2N} = F +
  y_\ell, \text{ and } y_0 = 0 \big\}.
\end{align*}
In particular, we observe that $y \in \Ys^\eps$ if and only if $y =
F x + u$ for some $u \in \Us^\eps$. 

The {\em stored energy} (per period) of an admissible
deformation $y \in \Ys^\eps$ is given by
\begin{displaymath}
\Es_\a(y) := \eps \sum_{\ell=-N+1}^N
\phi(y_\ell') +  \eps \sum_{\ell=-N+1}^N
\phi(y_{\ell-1}' + y_{\ell}'),  
\end{displaymath}
where $y_\ell' := \eps^{-1}(y_\ell-y_{\ell-1})$ and we note that
$y_{\ell-1}' + y_\ell' = \eps^{-1} (y_\ell-y_{\ell-2})$ is a
second-neighbour difference, and where $\phi \in C^3(0,+\infty)$ is a
Lennard-Jones type interaction potential. We assume throughout that
there exists $r_\ast >0$ such that $\phi$ is convex in $(0, r_\ast)$
and concave in $(r_\ast, +\infty)$.

Given a periodic dead load $f \in C^0(\R), f(x+1) = f(x)$, we define
the {\em external energy} (per period) by
\begin{equation}
  \label{eq:ext_energy}
  - \<f, u\>_\eps := - \eps \sum_{\ell=-N+1}^N  f_\ell u_\ell, 
\end{equation}
where $u = y - F x$ .  Thus, the {\em total energy} (per period) under
a deformation $y \in \Ys^\eps$ is given by
\begin{displaymath}
  \Etot_\a(y) := \Es_\a(y) - \<f, y - F x\>_\eps.
\end{displaymath}
We wish to compute
\begin{equation}
  y_\a \in \argmin \Etot_\a(\Ys^\eps),
  \label{Eq:LocalMinimaA}
\end{equation}
where $\argmin$ denotes the set of local minimizers.

\begin{remark}
  1. Since the internal energy is translation invariant, our choice
  $u_0 = 0$ (instead of the more common constraint $\sum_{\ell =
    -N+1}^N u_\ell = 0$) does not alter the problem but simplifies the
  treatment of external forces in Sections
  \ref{SectionResidualStoredEnergy} and \ref{sec:frc_est_sing}.

  2. The external energy \eqref{eq:ext_energy} can be thought of as the
  linearisation of a nonlinear external energy about the state $Fx$.

  3. We have included the external forces primarily to render the problem
  non-trivial. In realistic applications in 2D/3D, defects or forces
  applied on a boundary are the cause of deformation of the crystal,
  however, in 1D such complex behaviour cannot be described directly.
\end{remark}

\subsection{Notation for lattice functions}
\label{sec:notation_latticefcns}
Throughout, we identify a lattice function $v \in \R^\Z$ with its
continuous and piecewise affine interpolant with nodal values $v(\eps
\ell) = v_\ell$. Vice-versa, if $g \in C^0(\R)$ then we always
identify it with its vectors of nodal values $(g_\ell)_{\ell \in \Z} =
(g(\eps\ell))_{\ell \in \Z}$.

Let $\mathcal{D}$ be a subset of $\Z$. For a vector $v \in \R^\Z$, we define
\begin{align*}
\Vert v \Vert_{\ell^{p}_{\eps}(\mathcal{D})} :=  \left\{
\begin{array}{l l}
\Big(\sum_{\ell \in \mathcal{D}} \eps |v_\ell|^p
\Big)^{1/p},  & 1 \le p < \infty,\\
\max_{\ell \in \mathcal{D}} |v_\ell|,  & p = \infty.
\end{array} \right.
\end{align*}
If the label $\mathcal{D}$ is omitted, then we understand this to mean
$\mathcal{D} = \{-N+1, \dots, N\}$.

We define the first discrete derivatives $v_\ell' := (v_\ell -
v_{\ell-1}) / \eps$.  It is then straightforward to see that, if at
the same time we identify $v'$ with the pointwise derivative, then $\|
v' \|_{L^p(-1/2, 1/2)} = \Vert v' \Vert _{\ell^{p}_{\eps}}$.

We equip the space $\Us^\eps$ with the discrete Sobolev norm (recall
that $u_0 = 0$)
\begin{displaymath}
  \|v\|_{\Us^{1,2}} := \Vert v' \Vert _{L^2(-1/2, 1/2)} \quad \text{ for }
  v \in \Us^\eps.
\end{displaymath}
The norm on the dual $(\Us^\eps)^*$ is defined by
\begin{displaymath}
  \| T \|_{\Us^{-1,2}} := \sup_{\substack{v \in \Us^\eps \\ \| v
    \|_{\Us^{1,2}}=1 }} T[v].
\end{displaymath}

\def\per{\#}

\subsection{Finite element notation}
\label{sec:notation}
To construct the QC approximation in the next section, we first define
some convenient notation. Let $\Omega := [-1/2, 1/2]$ denote the
computational cell. We choose an atomistic region $\Omega_\a \subset
\Omega$, where atomistic accuracy is required, and we define the
continuum region by $\Omega_\c := \Omega \setminus \Omega_\a$. We will
also use the periodic extension of $\Omega_\c$, denoted by
$\Omega_\c^\per := (\Omega_\c + \Z)$.

We assume throughout that $\Omega_\a$ is an open interval $(L_\a,
R_\a)$ with $-1/2 < L_\a < R_\a < 1/2$. All our results (with the
exception of Section \ref{sec:frc_est_sing}) can be extended without
difficulty to the case when $\Omega_\a$ consists of a finite union of
open intervals.

Let $\{T_k^{h}\}_{k \in \Z}$ be a partition of $\R$ into closed
intervals with $T_k^{h} = [x^{h}_{k-1}, x^{h}_{k}]$, where $x^{h}_k >
x^{h}_{k-1}$ are the nodes of the partition. We assume, without loss
of generality, that $x_1^h$ is the left-most node and $x_K^h$ the
right-most nodes in the interval $(-1/2, 1/2]$.

The length of an element is denoted by $h_k := |T_k^{h}| := x^{h}_k -
x^{h}_{k-1}$. The space of continuous piecewise affine functions with
respect to the partition $\Ts^h$ is denoted by $\Ps_1(\Ts^h)$.

We assume throughout that the partition $\Ts^h$ has the following
properties:
\begin{itemize}
\item[(T1)] $\Ts^h$ is periodic: there exists $K \in \N$ such that
  $x_{k+K}^h = 1 + x_{k}^h$ for all $k \in \Z$.
\item[(T2)] $\Ts^h$ has atomistic resolution in $\Omega_\a$:
  $\Omega_\a \cap \eps \Z = \Omega_\a \cap \{ x_k^h \}_{k \in \Z}$.
\item[(T3)] The a/c interface points are finite
  element nodes: $\partial \Omega_\a \subset \{x^h_k\}_{k \in \Z}$. In
  particular, each element belongs entirely to either the atomistic or
  continuum region.
\item[(T4)] If $T^h_k \subset \Omega_\c^\per$ then $|T_k^h| = h_k \ge
  2 \eps$.
\end{itemize}
Property (T4) is not strictly required, but simplifies the analysis
and is not a significant restriction. Note also that we have not
required (except in the atomistic region) that finite element nodes
must be positioned on atomic sites. Although not necessary in 1D, it
somewhat simplifies mesh generation, and to some extend mimics the
fact that element edges or faces in 2D/3D cannot normally be aligned
with the underlying crystal lattice.

The finite element displacement and deformation spaces are defined,
respectively, by
\begin{align}
  \label{Def:QCDSpaceDisplace}
  \Us^h :=~& \big\{u_h \in \Ps_1(\Ts^h) : u_h(x+1) = u_h(x) \text{ and
  } u_h(0) = 0 \big\}, \quad 
  \text{and} \\ 
  \label{Def:QCDSpaceDeform}
  \Ys^h :=~& \big\{y_h \in \Ps_1(\Ts^h) : y_h - F x \in \Us^h \big\}.
\end{align}

For $g \in C^0(\R)$, we define the interpolation operator $I_h :
C^0(\R) \rightarrow \Ps_1(\Ts^h)$ by
\begin{equation}
  \label{Eq:CLinNodInterp}
  (I_h g)(x^h_k) := g(x^h_k) \qquad \forall k \in \Z.
\end{equation}
We note that $I_h : \Us^\eps \to \Us^h$. 

For future reference, we also define the micro-elements $T_\ell^\eps
:= ((\ell-1)\eps, \ell \eps)$ for $\ell \in \Z$. Analogously, we
define $I_\eps$ to be the nodal interpolant with respect to the
atomistic grid. We will require this interpolant since the mesh nodes
$\{x_k^h\}_{k \in\Z}$ do not necessarily coincide with lattice sites.

\subsection{QC Approximation}
\label{subsec:QCApprox}
The QC approximation we analyze in this paper is the 1D variant of the
ACC method described in~\cite{Shapeev2010a}. (An earlier variant of
the idea was described in \cite{Ortner:arXiv:0911.0671} and a similar
construction in \cite{LiLuskin2011a}. We focus on the formulation
proposed in \cite{Shapeev2010a} since it can be readily generalised to
2D.)

The idea of the ACC method is based on the splitting of interaction
bonds. A bond is an open interval $b = (\ell\eps, (\ell+j)\eps)$ for
$\ell \in \Z$ and $j \in \{1,2\}$ (since we consider only first and
second neighbour interactions). Since our computational domain is
$(0,1]$ the set of bonds over which the atomistic energy is defined is
given by
\begin{equation}
  \Bs:= \big\{ (\ell\eps, (\ell+j)\eps): j = 1,2; \ell =
  -N+1, \ldots, N \big\}.
\end{equation}
For each bond $b = (\ell\eps, (\ell+j)\eps)$ we define $r_b := j$.

For any open interval $\omega = (L_\omega, R_\omega)$ (e.g., for a
bond) with length $|\omega| := R_\omega - L_\omega > 0$ we define the
finite difference operator
\begin{equation}
  D_\omega v := \frac{v(R_\omega)-v(L_\omega)}{|\omega|} \qquad
  \text{for } v \in C^0(\R).
\end{equation}
Note that, with this notation, we can rewrite $\Es_\a(y) = \eps
\sum_{b \in \Bs} \phi\big( r_b D_b y \big)$.

For each bond $b$ and deformation field $y \in W^{1,\infty}(\R)$ we
define its atomistic and continuum energy contributions to the stored
energy, respectively, by
\begin{align*}
  a_b(y) := \frac{|b \cap \Omega_\a|}{|b|} \phi\big(r_b D_{b \cap
    \Omega_\a} y \big) \quad \text{and} \quad 
  c_b(y) := \frac{1}{|b|} \int_{b \cap \Omega_\c^\per} \phi\big(r_by'(x)\big)\dx.
\end{align*}
If $|b \cap \Omega_\a| = 0$ then $D_{b \cap \Omega_\a} y$ is
ill-defined; in that case we define $a_b \equiv 0$. The QC energy (the
ACC energy of \cite{Shapeev2010a}) of a deformation field $y \in
W^{1,\infty}(\R)$ is now defined by
\begin{equation}
  \Es_{\qc}(y) := \eps\sum_{b \in \Bs} \big[ a_b(y)+c_b(y) \big].
  \label{QCStoredEnergy-1}
\end{equation}
The key property of this definition is that it satisfies the patch
test \cite[Section 3.3]{Shapeev2010a},
\begin{equation}
  \Es'_{\a}(Fx)[v_\eps] = \Es'_{\qc}(Fx)[v_h] = 0 \quad \text{for all }  v_\eps \in
  \Us^\eps \text{ and } v_h \in \Us^h, \quad \forall F > 0.
\end{equation}

The {\em external energy} (per period) is given by
\begin{equation}
  - \<f,u_h\>_h :=  - \int_\Omega I_h (f u_h) \dx. 
\end{equation}
Note that, in the atomistic region, this reduces to the same form as
\eqref{eq:ext_energy}. The {\em total energy} (per period) of a
deformation $y_h \in \Ys^h$ is then given by
\begin{displaymath}
  \Etot_{\qc}(y_h) := \Es_{\qc}(y_h) - \<f, y_h-F x\>_h.
\end{displaymath}
In the QC approximation we seek
\begin{equation}
  \label{LocalMinimumQC}
  y_{\qc} \in \argmin \,\Etot_{\qc}(\Ys^h).
\end{equation}

\begin{remark}
  It is initially not obvious why the formulation
  \eqref{QCStoredEnergy-1} should reduce the complexity of the
  computation of $y_\qc$ over that of $y_\a$, since $\Es_\qc$ is still
  written as a summation over all bonds. However, one can readily
  check (see \cite{Shapeev2010a} for the details), that
  \begin{displaymath}
    \Es_{\qc}(y_h) = \sum_{b \in \Bs} a_b(y_h) + \int_{\Omega_\c} W(y_h') \dx,
  \end{displaymath}
  where $W(r) := \phi(r) + \phi(2r)$ is the Cauchy--Born stored energy
  function. This formulation requires only a sum over all bonds within
  the atomistic region, and a standard finite element assembly
  procedure to compute the energy contribution from the continuum
  region. Thus, the evaluation of $\Es_\qc$ has only a computational
  complexity proportional to $\#\Ts^h$.

  Analogous results hold in 2D \cite{Shapeev2010a}; a more involved
  assembly procedure is required to make the ACC method efficient in
  3D \cite{Shapeev_3D}.
\end{remark}

\section{Residual Analysis}
\label{Sec:Residual_Analysis}

\subsection{The residual}
Let $y_\a$ be a solution of the atomistic problem
\eqref{Eq:LocalMinimaA}. It is straightforward to see that, if
$y_\a'>0$, then $\Es_{\a}$ is differentiable at $y_\a$ and hence the
first order optimality condition for \eqref{Eq:LocalMinimaA} is
satisfied:
\begin{align}
  \label{VariationalA}
&   \Es'_\a(y_\a)[v] = \<f,v\>_\eps \qquad \forall v \in \Us^\eps, \\
  \label{VariationalAStore}
  & \text{where} \quad \Es'_\a(y_\a)[v] = \eps \sum_{b
    \in \Bs} \phi'(r_b D_b y_\a) r_bD_b v.
\end{align} 

Similarly, if $y_{\qc}$ is a solution of the QC problem
\eqref{LocalMinimumQC} with $y_{\qc}' > 0$ on $[-1/2,1/2]$, then it
satisfies the corresponding first order optimality condition
\begin{align}
\label{VariationalQC}
& \hspace{-3cm}  \Es'_{\qc}(y_{\qc})[v_h] = \<f,v_h\>_h \qquad \forall v_h \in \Us^h,
\\
\label{VariationalQCStore}
\text{where} \quad 
\Es'_{\qc}(y_h)[v_h] &= \sum_{b \in \Bs} \big(
a'_b(y_h)[v_h]+c_b'(y_h)[v_h] \big) \\
\nonumber
&= \sum_{b \in \Bs} |b\cap \Omega_\a| \phi'(r_b D_{b \cap \Omega_\a} y_h) D_{b \cap
  \Omega_\a} v_h  + \sum_{b \in \Bs} \int_{b \cap
  \Omega_\c^\per} \phi'({r_b y'_h(x)}) v_h' \dx.
\end{align}


Let $y_{\qc} \in \Ys^h$ be a solution of \eqref{LocalMinimumQC}, then
we define its residual $R \in (\Us^\eps)^*$ by
\begin{align}
  \label{definitionTotalResidual}
  R[v] := E'_{\a}(y_{\qc})[v],
\end{align}
%
Using
\eqref{VariationalQC} we can rewrite this as
\begin{align*}
   R[v] =~& E_{\a}'(y_{\qc})[v] - E'_{\qc}(y_{\qc})[I_h v] \nonumber\\
       =~&  \big\{\Es'_{\a} (y_{\qc})[v]-\Es'_{\qc}(y_{\qc})[I_h
       v]\big\} 
       - \big\{ \<f,v\>_\eps - \<f, I_h v\>_h \big\}  \notag \\
       =:~& R_{\rm int}[v] + R_{\rm ext}[v],
\end{align*}
where we call $R_{\rm int}$ the internal residual and $R_{\rm ext}$
the external residual. We will bound them separately in the next two
sections.

\subsection{Estimate for the internal residual}
\label{SectionResidualStoredEnergy}
In this section, we analyze the internal residual $R_{\rm int}$. To
state the main theorem, we define the index set of all nodes in
the continuum region and a/c interface (recall that $\Omega_\c$ is
closed),
\begin{equation}
  \label{Set:K_c}
  \Ks_\c := \big\{ k \in \{1,\ldots,K\} : x_k^h \in \Omega_\c \big\}.
\end{equation}

\begin{theorem}
  \label{Theo:ResidualStoredEnergy}
  Let $y_h \in \Ys^h$ such that $y_h' > 0$ and $R_{\rm int}[v] =
  \Es'_{\a}(y_h)[v] - \Es'_{\qc}(y_h)[I_h v]$, then
  \begin{align}
    \label{ResidualStoredEnergy}
    & \big\| R_{\rm int} \big\|_{\Us^{-1,2}} \le
    \Big( 3 \sum_{k \in \mathcal{K}_\c} \eta_k^2 \Big)^{\frac{1}{2}}
    =: \eta(y_h), \qquad \text{where} \\
    %
    \notag
    & \eta_k^2 := \!\!\!\sum_{\substack{b \in \Bs \\ x_k^h \in {\rm int}(b)}}
      \Big( \big\|  \phi'(r_b D_b y_h)
      - \phi'(r_b D_{b \cap \Omega_\a} y_h) \big\|_{L^2(b \cap
     \Omega_\a)}^2 
   + \big\| \phi'(r_b D_b y_h)
   - \phi'(r_b y_h') \big\|_{L^2(b\cap \Omega_\c^\per)}^2 \Big).
  \end{align}
\end{theorem}

\begin{remark}
  1. The expressions for $\eta_k$ are reminiscent of the flux (or
  stress) jump terms that occur in the classical residual error
  analysis for partial differential equations. The origin in our case,
  is somewhat different however, and results only from the model
  approximation and not the finite element coarsening.

  2. With some additional work, the form of the estimates $\rho_k$ can
  be turned into element contributions and further simplified by
  computing more explicit representations.  
\end{remark}

\begin{proof}
  Let $v_h := I_h v$. From
  \eqref{VariationalAStore} and \eqref{VariationalQCStore} we obtain
  \begin{align}
    \label{ResidualStore1}
    R_{\rm int}[v]
    =& \sum_{b
      \in \Bs} \Big\{ |b| \phi'(r_b D_b y_h) D_b v-
    |b\cap \Omega_\a| \phi'(r_b D_{b \cap \Omega_\a} y_h) D_{b \cap
      \Omega_\a} v_h   - \int_{b \cap
      \Omega_\c^\per} \phi'(r_b y_h') v_h' \dx \Big\}.
  \end{align}
  We subtract and add the terms
  \begin{displaymath}
    \sum_{b \in \Bs}|b\cap \Omega_\a| \phi'(r_b D_{b \cap \Omega_\a} y_h)  D_{b \cap
      \Omega_\a}v \quad \text{ and } \quad \sum_{b \in \Bs} \int_{b \cap
      \Omega_\c^\per} \phi'({r_b y'_h(x)}) v' \dx
  \end{displaymath}
  to split $R_{\rm int}$ into three components,
  \begin{align}
    \notag
    & \hspace{-1cm} R_{\rm int}[v] \\
    =~&\sum_{b
      \in \Bs} \bigg\{ |b| \phi'(r_b D_b y_h) D_b v -|b\cap \Omega_\a| \phi'(r_b D_{b \cap \Omega_\a} y_h)  D_{b \cap
      \Omega_\a} v  - \int_{b \cap
      \Omega_\c^\per} \phi'({r_b y'_h(x)}) v' \dx \bigg\} \nonumber\\
    \notag
    & +\sum_{b \in \Bs}
    |b\cap \Omega_\a| \phi'(r_b D_{b \cap \Omega_\a} y_h) \big(  D_{b
      \cap \Omega_\a} v - D_{b \cap \Omega_\a} v_h \big) \\
    \notag
    & + \sum_{b \in \Bs} \int_{b \cap
      \Omega_\c^\per} \phi'({r_b y'_h(x)}) \big(v' - v_h' \big) \dx \notag\\
    \label{ResidualStore2}
    =:~& R_1[v] + R_2[v] + R_3[v].
  \end{align}
  We will show that $R_2 \equiv R_3 \equiv 0$ and estimate $R_1$.

  For $R_2$ this follows simply from the fact that $v_h = v$ in
  $\Omega_\a$ and hence $D_{b \cap \Omega_\a} v = D_{b \cap \Omega_\a}
  v_h$ for all $b \in \Bs$ such that $|b \cap \Omega_\a| > 0$.

  For $R_3$, following the computations in \cite[Section
  3.2]{Shapeev2010a}, we obtain
  \begin{displaymath}
    R_3[v] = \int_{\Omega_\c} W'(y_h') (v' - v_h') \dx,
  \end{displaymath}
  where $W(r) = \phi(r)+\phi(2r)$.  Since $v(x_k^h) = v_h(x_k^h)$ for
  all $k \in \Z$ and since $W'(y_h')$ is constant on each element
  $T_k^h$ it follows that $R_3 \equiv 0$. 

  Finally, we turn to the analysis of $R_1 = \sum_{b \in \Bs} R_1^b$,
  where we define
  \begin{align*}
    R_1^b[v] :=~& |b| \phi'(r_b D_b  y_h) D_b v - |b\cap \Omega_\a| \phi'(r_b D_{b \cap \Omega_\a} y_h)  D_{b \cap
      \Omega_\a} v 
    - \int_{b \cap
      \Omega_\c^\per} \phi'({r_b y'_h(x)}) v' \dx.
  \end{align*}
  Using the fact that $|\omega| D_\omega v = \int_\omega v' \dx$ we obtain
  \begin{align*} 
    \notag 
    R_1^b[v] =~& \int_{b} \phi'(r_b D_b y_h) v' \dx 
    - \int_{b \cap \Omega_\a} \phi'(r_b D_{b \cap \Omega_\a} y_h) v' \dx
    - \int_{b \cap\Omega_\c^\per} \phi'({r_b y'_h(x)}) v' \dx \\
    %
    =~& \int_{b \cap \Omega_\a} \big[\phi'(r_b D_b y_h) - \phi'(r_b
    D_{b \cap \Omega_\a} y_h)\big] v' \dx 
    + \int_{b \cap\Omega_\c^\per} \big[ \phi'(r_b D_b y_h) - \phi'(r_b y_h')\big] v' \dx. 
 \end{align*}

 If $b \subset \Omega_\a$, then $b \cap \Omega_\a = b$ and $|b \cap
 \Omega_\c| = 0$, and hence $R_1^b = 0$. Similarly, if $b \subset T_k^h
 \subset \Omega_\c^\per$, then $D_b y_h = y_h'|_{T_k^h}$ and $|b \cap
 \Omega_\a| = 0$ and hence $R_1^b = 0$. Thus, we observe that only
 bonds crossing continuum element boundaries, or the
 atomistic/continuum interface, contribute to the residual. These are
 precisely the points contained in $\Ks_\c$. In
 particular, we obtain
 \begin{equation}
   \label{eq:residual_R1_2}
   R_1[v] = \sum_{k \in \Ks_\c} \sum_{\substack{b \in \Bs \\ x_k^h \in {\rm
         int}(b)}} R_1^b[v],
 \end{equation}
 where we used the fact that no bond can cross more than one point
 $x_k^h \in \Ks_\c$ due to our assumption that all elements have at
 least length $2\eps$.

 From the definition of $R_1^b$, and applying two Cauchy--Schwarz
 inequalities, it is straightforward to estimate
 \begin{align*}
   \big|R_1^b[v] \big| \leq \Big(\,& \big\| \phi'(r_b D_b y_h) - \phi'(r_b
    D_{b \cap \Omega_\a} y_h) \big\|_{L^2(b \cap \Omega_\a)}^2 \\
    & \qquad +
    \big\| \phi'(r_b D_b y_h) - \phi'(r_b y_h') \big\|_{L^2(b \cap
      \Omega_\c^\per)}^2 \Big)^{1/2} \| v' \|_{L^2(b)},
 \end{align*}
 and after applying another Cauchy--Schwarz inequality,
 \begin{align*}
   \Big|\sum_{\substack{b \in \Bs \\ x_k^h \in {\rm int}(b)}} R_1^b[v]
   \Big| \leq\,& \eta_k \,\Big( 
   \sum_{\substack{b \in \Bs \\ x_k^h \in {\rm int}(b)}} \| v'
   \|_{L^2(b)}^2 \Big)^{1/2},
 \end{align*}
 where $\eta_k$ is defined in \eqref{ResidualStoredEnergy}. 

 Combing our foregoing estimates, we arrive at
 \begin{align*}
   R_1[v] \leq\,& \sum_{k \in \Ks_\c} \eta_k \,\Big( 
   \sum_{\substack{b \in \Bs \\ x_k^h \in {\rm int}(b)}} \| v'
   \|_{L^2(b)}^2 \Big)^{1/2} \\
   \leq\,& \Big(\sum_{k \in \Ks_\c} \eta_k^2 \Big)^{1/2}\, \Big(
   \sum_{k \in \Ks_\c}  \sum_{\substack{b \in \Bs \\ x_k^h \in {\rm int}(b)}} \| v'
   \|_{L^2(b)}^2 \Big)^{1/2},
 \end{align*}
 and we are only left to estimate the sums involving the test
 function.

 To that end, we simply note that, due to (T4), for any fixed point $x
 \in (-1/2, 1/2]$, the maximal number of bonds appearing in the sum on the
 left-hand side below and crossing $x$ is three; hence,
 \begin{displaymath}
   \sum_{k \in \Ks_\c}  
   \sum_{\substack{b \in \Bs \\ x_k^h \in {\rm int}(b)}} \| v'
   \|_{L^2(b)}^2 \leq 3 \| v' \|_{L^2}^2.
 \end{displaymath}
 This concludes the proof.
\end{proof}

\subsection{Estimate of the external residual}
\label{SectionResidualExternalForce} 
We now turn to the estimate of the residual of the external energy,
which was defined as
\begin{equation}
  R_{\rm ext}[v] = \<f, v\>_\eps - \<f,I_h v\>_h = \int_\Omega \big[I_\eps(f
  v) - I_h (f I_h v) \big] \dx.
\label{Eq:ResGraExEnOrig}
\end{equation}
We outline the key points of the argument for estimating $R_{\rm
  ext}$, before stating the result.

We define a slightly extended continuum region,
\begin{displaymath}
  \tilde\Omega_\c := \bigcup \big\{ T_\ell^\eps : |T_\ell^\eps \cap \Omega_\c| > 0 \big\};
\end{displaymath}
then $I_\eps(f v) = I_h(f I_h v)$ in $\Omega_\a \setminus
\tilde{\Omega}_\c$, and therefore
\begin{align} 
  \notag
  R_{\rm ext}[v] =\,& \int_{\tilde\Omega_\c} \big[ I_\eps(f v) - I_h
  (f I_hv) \big] \dx \\
  \notag
  =\,& \int_{\tilde\Omega_\c} \big[ I_\eps(f v) - f v \big] \dx +
  \int_{\tilde\Omega_\c} \big[ f v - f I_h v \big] \dx +
  \int_{\tilde\Omega_\c} \big[ f I_h v - I_h (f I_h v) \big] \dx \\
  \label{eq:Rext:decomp}
  =:\,& R_1[v] + R_2[v] + R_3[v].
\end{align}
The three terms can be estimated using standard interpolation error
results, hence we only give a brief outline of the proof of the
resulting bound.

\begin{proposition}
  \label{Theo:RisidualExternalForce}
  Let $f \in C^2(\tilde{\Omega}_c)$, then 
  \begin{equation}
    \label{ResidualExternalForce}
    \|R_{\rm ext}\|_{\Us^{-1,2}} \le \eta^f + \eta^q,
  \end{equation}
  where the error due to external forces $\eta^f$ and the ``quadrature
  error'' $\eta^q$ are, respectively, defined as follows: for $* \in
  \{f, q\}$, 
  \begin{align*}
    (\eta^*)^2 :=\,& \sum_{\substack{k \in \{1,\dots, K\} \\ T_k^h
        \subset \tilde\Omega_\c}} (\eta_k^*)^2, \\
    \text{where} \qquad (\eta_k^f)^2 :=\,& \smfrac{h_k^2}{\pi^2} \| f \|_{L^2(T_k^h)}^2,
    \\
    \text{and}  \qquad (\eta_k^q)^2 :=\,& (\eps^4 + h_k^4)
    \| f' \|_{L^2(T_k^h)}^2 + {\smfrac{(\eps^4 + h_k^4)}{4 \pi^2}} \| f''\|_{L^2(T_k^h)}^2.
  \end{align*}
  %
\end{proposition}

\begin{remark}
  1. Note that there is an error contribution from the atomistic
  region, due to the fact that in the elements touching the a/c
  interface, the ``quadrature'' approximation of the external forces
  is not exact. For the purpose of mesh refinement, we count this
  error towards the neighbouring elements in the continuum region.

  2. An alternative residual estimate that does not use $f \in
  C^2(\tilde\Omega_\c)$, but only the discrete setting, is presented
  in \cite{Wang:arXiv:1112.5480}. This requires a much more involved
  argument.
\end{remark}

\begin{proof}
  From \eqref{eq:Rext:decomp} we obtain
  \begin{displaymath}
    R_{\rm ext}[v] = R_1[v] + R_2[v] + R_3[v].
  \end{displaymath}
  Applying standard interpolation error estimates (see, e.g.,
  \cite{BrennerScott, BraessFEM}) on elements $T_\ell^\eps, T_k^h$, we
  obtain
  \begin{align*}
    \int_{T^\eps_\ell} \big[ I_\eps(f v) - f v \big] \dx \leq~& \smfrac{1}{4} \| \eps^2f''\|_{L^2(T^\eps_\ell)}\|v\|_{L^2(T^\eps_\ell)} + 
    \smfrac{1}{2} \| \eps^2 f' \|_{L^2(T^\eps_\ell)} \|v'\|_{L^2(T^\eps_\ell)},  \\
    \int_{T^h_k} \big[ f v - f I_h v \big] \dx \leq~& \smfrac{1}{\pi} \| h_k f
    \|_{L^2(T^h_k)} \| v' \|_{L^2(T^h_k)}, \quad \text{and} \\
    \int_{T^h_k} \big[ I_h(f v) - f I_hv \big] \dx \leq~& \smfrac{1}{4} \| h_k^2f''\|_{L^2(T^h_k)} \|I_h v\|_{L^2(T^h_k)} + 
    \smfrac{1}{2} \| h_k^2 f' \|_{L^2(T^h_k)} \|I_h v'\|_{L^2(T^h_k)}.
  \end{align*}
  Summing over all elements, applying the Cauchy--Schwarz inequality,
  and defining $h(x) := h_k$ for $x \in T_k^h$,
  \begin{align*}
    R_1[v] 
    \leq\,& 
    \smfrac{\eps^2}{4} \| f'' \|_{L^2(\tilde\Omega_\c)} \, \| v \|_{L^2} +
    \smfrac{\eps^2}{2} \| f' \|_{L^2(\tilde\Omega_\c)}  \, \| v' \|_{L^2},  \\
    R_2[v] \leq\,&
    \smfrac{1}{\pi} \| h f
    \|_{L^2(\tilde\Omega_\c)} \, \| v' \|_{L^2}, \qquad \text{and}   \\
    R_3[v] \leq\,& 
    \smfrac{1}{4} \| h^2 f'' \|_{L^2(\tilde\Omega_\c)} \, \| I_h v
    \|_{L^2} + \smfrac{1}{2} \| h^2 f' \|_{L^2(\tilde\Omega_\c)} \, \|
    I_h v' \|_{L^2}.
  \end{align*}
  We now use the estimates (which exploit the fact that $I_h v'$ is
  the $L^2$-orthogonal projection of $v'$ onto piecewise constants)
  \begin{align*}
    & \| v \|_{L^2} \leq {\smfrac{1}{\pi}} \| v' \|_{L^2}, 
    \quad
    \| I_h v' \|_{L^2} \leq \| v' \|_{L^2}, \\
    \text{and} \quad& 
    \| I_h v \|_{L^2} \leq \smfrac{1}{\pi} \| I_h v' \|_{L^2} \leq
    \smfrac{1}{\pi} \| v' \|_{L^2},
  \end{align*}
  to deduce that
  \begin{align*}
    R_{\rm ext}[v] =\,& R_1[v] + R_2[v] + R_3[v]
    \\
    \leq\,&\smfrac{1}{\pi} \| h f \|_{L^2(\tilde\Omega_\c)}  \| v' \|_{L^2} + \Big( \smfrac{1}{4 \pi^2} \| (\eps^2+h^2) f'' \|_{L^2(\tilde\Omega_\c)}^2 +
    \| (\eps^2+h^2) f' \|_{L^2(\tilde\Omega_\c)}^2
    \Big)^{1/2} \| v' \|_{L^2}.
  \end{align*}
  The result follows by splitting the norms inside the brackets over
  elements.
\end{proof}

\subsection{External residual estimate for singular forces}
\label{sec:frc_est_sing}
In our numerical experiments in Section \ref{Numerics} we shall employ
an external force that behaves like $|f(x)| \sim |x|^{-1}$ near $x =
0$ (we use the ``singularity'' in the force to mimic a defect). Let us
suppose that we also have $|f'(x)| \sim |x|^{-2}$ and $|f''(x)| \sim
|x|^{-3}$ near the origin. We now give a formal motivation why the
quadrature estimates employed in Proposition
\ref{Theo:RisidualExternalForce} are inadequate in this situation.

Applying the quadrature estimates to such a force field,  we obtain
\begin{align*}
  \eta_k^q \approx h_k^2 \| f' \|_{L^2(T_k^h)} + h_k^2 \| f''
  \|_{L^2(T_k^h)} 
  \approx h_k^{5/2} |x_k^h|^{-3/2} + h_k^{5/2} |x_k^h|^{-5/2}.
\end{align*}
We notice that, for $T_k^h$ near the origin, $\| f'' \|_{L^2(T_k^h)}
\approx |x_k^h|^{-1} \| f' \|_{L^2(T_k^h)}$. Moreover, the quadrature
estimate is $O(1)$ and cannot be controlled.  By contrast,
\begin{displaymath}
  \eta_k^f \approx h_k^{3/2} |x_k^h|^{-1}, 
\end{displaymath}
from which we conclude that $h_k^2 \| f' \|_{L^2(T_k^h)}$ is dominated
by $\eta_k^f$, but that $\eta_k^f$ is itself dominated by $h_k^2 \|
f'' \|_{L^2(T_k^h)}$.

The origin of this undesirable effect is the (ab-)use of the
Poincar\'e inequality in the proof of Proposition
\ref{Theo:RisidualExternalForce}. In the remainder of this section, we
shall remedy the situation by replacing the standard Poincar\'e
inequality with a weighted variant. This approach is inspired by
\cite{OL_Acta}.

\begin{lemma}
  \label{th:weighted_poincare}
  Let $\tilde\Omega_\c$ be defined as in Section
  \ref{SectionResidualExternalForce}, and let $w(x) := x \log^2(x)$,
  then
  \begin{displaymath}
    \| w^{-1} v \|_{L^2(\tilde\Omega_\c)} \leq \smfrac{1}{\log 2} \| v'
    \|_{L^2} \qquad \forall v \in H^1(\Omega), v(0) = 0,
  \end{displaymath}
\end{lemma}
\begin{proof}
  We begin by noting that
  \begin{displaymath}
    |v(x)| \leq |x|^{1/2} \| v' \|_{L^2} \quad \text{for all } x \in \Omega.
  \end{displaymath}
  Hence, we can estimate
  \begin{displaymath}
    \int_{\tilde{R}_\a}^{1/2} | w^{-1} v(x)|^2 \dx \leq  \| v'
    \|_{L^2(0, 1/2)}^2 \int_{\tilde{R}_\a}^{1/2} \frac{1}{x
      \log^2(2x)} \dx.
  \end{displaymath}
  Since $(\log^{-1}(x))' = - (x \log^2(x))^{-1}$ we obtain
  \begin{displaymath}
    \int_{\tilde{R}_\a}^{1/2} | w^{-1} v(x)|^2 \dx \leq
    \smfrac{1}{\log(2)} \|v' \|_{L^2(0, 1/2)}^2.
  \end{displaymath}
  Applying an analogous argument in the left half of the domain, we
  obtain the stated estimate.
\end{proof}

We now apply this estimate to obtain an alternative external residual
estimate.

\begin{proposition}
  \label{Theo:RisidualExternalForce_Sing}
  Let $f \in C^2(\tilde{\Omega}_c)$, then 
  \begin{equation}
    \label{ResidualExternalForce_sing}
    \|R_{\rm ext}\|_{\Us^{-1,2}} \le \eta^f + \hat{\eta}^q,
  \end{equation}
  where $\eta^f$ is defined in \eqref{ResidualExternalForce} and
  $\hat{\eta}^q$ is defined as follows:
  \begin{align*}
    (\hat{\eta}^q)^2 :=\,& \sum_{\substack{k \in \{1,\dots, K\} \\ T_k^h
        \subset \tilde\Omega_\c}} (\hat{\eta}^q_k)^2, \\
    \text{where}  \qquad (\hat\eta_k^q)^2 :=\,& (\eps^4 + h_k^4)
    \| f' \|_{L^2(T_k^h)}^2 + \smfrac{\eps^4 + h_k^4}{\log^2 2} \| w f''\|_{L^2(T_k^h)}^2,
  \end{align*}
  and where $w$ is defined in Lemma \ref{th:weighted_poincare}.
\end{proposition}
\begin{proof}
  We again use the splitting \eqref{eq:Rext:decomp} to obtain
  \begin{displaymath}
    R_{\rm ext}[v] = R_1[v] + R_2[v] + R_3[v].
  \end{displaymath}
  The residual term $R_2[v]$ is estimated in the same way as in the
  proof of Proposition \ref{Theo:RisidualExternalForce}, and gives
  rise to the term $\eta^f$ in the estimate.

  We show only the modified estimate for $R_3[v]$, since the estimate
  for $R_1[v]$ is analogous. Applying again a standard interpolation
  error estimate, we obtain
  \begin{align}
    \big|R_3[v]\big| \leq \smfrac14 \| h^2 (f I_h v)''
    \|_{L^1(\tilde\Omega_\c)} 
    \label{eq:singfrc_10}
    \leq \smfrac12 \| h^2 f' I_h v' \|_{L^1(\tilde\Omega_\c)} +
    \smfrac14 \| h^2 f'' I_h v \|_{L^1(\tilde\Omega_\c)}.
  \end{align}
  The term $\smfrac12 \| h^2 f' I_h v' \|_{L^1(\tilde\Omega_\c)}$ can
  be treated in the same way as in the proof of Proposition
  \ref{Theo:RisidualExternalForce}.

  To estimate the second term on the right-hand side of
  \eqref{eq:singfrc_10} we insert the weighting function $w$ defined
  in Lemma \ref{th:weighted_poincare} and then apply the weighted
  Poincar\'e inequality:
  \begin{align*}
    \smfrac14 \| h^2 f'' I_h v \|_{L^1(\tilde\Omega_\c)} =\,& 
     \smfrac14 \big\| (h^2 w f'') ( w^{-1} I_h v)
     \big\|_{L^1(\tilde\Omega_\c)}  \\
    \leq\,& \smfrac14 \| h^2 w f'' \|_{L^2(\tilde\Omega_\c)}\, \|
    w^{-1} I_h v \|_{L^2(\tilde\Omega_\c)} \\
    \leq\,& \smfrac{1}{4 \log 2} \| h^2 w f'' \|_{L^2(\tilde\Omega_\c)}\, \|
    I_h v' \|_{L^2}.
  \end{align*}

  By continuing to argue as in the proof of Proposition
  \ref{Theo:RisidualExternalForce} we obtain the stated estimate.
\end{proof}

We can now revisit the issue of relative magnitude of the various
contributions to the residual estimate for the case where $|f(x)| \sim
|x|^{-1}$, $|f'(x)| \sim |x|^{-2}$ and $|f''(x)| \sim |x|^{-3}$. Note
that the effect of the weighting function is that $|w(x) f''(x)| \sim
|\log^2(x)| |x|^{-2}$ which is now comparable to $|f'(x)|$ up to a log
factor.

More precisely, suppose that $T_k^h$ is near the origin, then we now
obtain
\begin{displaymath}
  \hat{\eta}_k^q \approx h_k^{5/2} |x_k^h|^{-3/2} + h_k^{5/2}
  |x_k^h|^{-3/2} \log^2 (x_k^h) .
\end{displaymath}
In particular, we observe that in the new external residual estimate,
the quadrature error is dominated by the main error term $\eta^f$,
which is the same as in the standard estimate given in Proposition
\ref{SectionResidualExternalForce}. Thus, in our numerical algorithms
presented in Section \ref{Numerics} we will be justified in neglecting
the effect of the quadrature errors.

\section{Stability}
\label{Sec:Stability}
Stability of the exact (i.e., the atomistic) model is the second key
ingredient for deriving an a posteriori error bound. Our aim is to
prove coercivity (or, positivity) of the atomistic Hessian at the QC
solution $y_{\qc}$:
\begin{equation}
  E''_{\a}(y_{\qc}) [v,v] \ge c_\a(y_{\qc}) \| v' \|_{L^2}^2 \qquad
  \forall v \in \Us,
\label{Stability-1}
\end{equation}
for some constant $c_\a(y_{\qc}) >0$, where the Hessian operator of the
atomistic model is given by
\begin{equation*}
E''_{\a}(y)[v,v] = \eps \sum_{\ell=-N+1}^N \phi''(y_\ell')|v_\ell'|^2 + \eps
\sum_{\ell=-N+1}^N \phi''(y_\ell' + y_{\ell+1}') |v_\ell'
+ v'_{\ell+1}|^2.
\end{equation*} 
Following \cite{Dobson:arXiv:0905.2914, Ortner:arXiv:0911.0671} we
note that the 'non-local' Hessian terms $|v'_\ell + v'_{\ell+1}|^2$
can be rewritten in terms of the 'local' terms $|v'_\ell|^2$ and
$|v'_{\ell+1}|^2$ and a strain-gradient correction,
\begin{equation*}
|v'_\ell + v'_{\ell+1}|^2 = 2|v'_\ell|^2+2
|v'_{\ell+1}|^2 - \eps^2 |v''_\ell|^2. 
\end{equation*}
Using this formula, we can rewrite the Hessian in the form
\begin{equation*}
E_{\a}''(y)[v,v] = \eps \sum_{\ell=-N+1}^N A_\ell |v_\ell'|^2 + \eps
\sum_{\ell = -N+1}^{N} B_\ell |v_\ell''|^2,
\end{equation*}
where
\begin{align}
&A_{\ell}(y) := \phi''(y_\ell')+2\phi''(y_{\ell-1}' + y_\ell') +
2\phi''(y_\ell'+y_{\ell+1}')
\label{StabilityConstant-1}
\\
& B_{\ell}(y) := -\phi''(y'_{\ell}+y'_{\ell+1}) \nonumber.
\end{align}

Recall our assumption in \S\ref{subsec:AModel} that $\phi$ is convex
in $(0, {r_\ast})$ and concave in $({r_\ast}, +\infty)$. For typical
interactions such as Lennard-Jones or Morse potentials one generally
observes that $y_\ell' \geq r_\ast / 2$, hence we shall assume this
throughout.  As a result of this assumption, and the properties of
$\phi$, we have $B_\ell \ge 0 \ \forall \ell \in \Z$.

As an immediate consequence we obtain the following lemma, which gives
sufficient conditions for the stability of the atomistic hessian
evaluated at the QC solution.

\begin{proposition}
  \label{th:stab_lemma}
  Let $y_{\qc} \in \Ys^\eps$ satisfy $\min_{\ell}(y_\qc)'_{\ell} \ge {r_\ast}/2$, then
  \begin{displaymath}
    E''_{\a}(y_{\qc})[v,v] \ge A_{\ast}(y_{\qc}) \| v'
    \|_{\ell_{\eps}^2}^2 \qquad \forall v \in \Us, \quad
    \text{where} \quad
    A_\ast(y_{\qc}) := \min_{\ell = 1, \dots, N} A_\ell(y_{\qc}),
  \end{displaymath}
  and the coefficients $A_\ell(y_{\qc})$ are defined in
  \eqref{StabilityConstant-1}.
\end{proposition}


\begin{remark}
  Since the minimum in the definition of $A_*$ is taken over $2N$
  lattice sites, it appears at first glance that $A_*$ is expensive to
  evaluate. However, exploiting the fact that $y_\qc$ is piecewise
  affine, one can evaluate $A_*$ in $O(K)$ operations: 

  {\it Case i: } If ${\rm dist}(\eps (\ell-1/2), \{ x_k^h\}) < {\frac{3}{2}\eps}$
  then we evaluate $A_\ell(y_\qc)$ using
  \eqref{StabilityConstant-1}. There are $O(K)$ lattice sites of this
  type.

  {\it Case ii: } If ${\rm dist}(\eps (\ell-1/2), \{ x_k^h \}) \geq
  {\frac{3}{2}\eps}$ then
  \begin{displaymath}
    A_\ell(y_\qc) = \phi''(y_h'|_{T_k^h}) + 4\phi''(2y_h'|_{T_k^h}),
  \end{displaymath}
  that is, we only need to evaluate this formula once for each
  element.
\end{remark}

\medskip

\subsection{Estimates for the hessian}
Before we present our main theorems, we state two useful auxiliary
results: a local bound and a local Lipschitz bound on $E''_{\a}$. The
proofs are straightforward and are therefore omitted.

\begin{lemma}
  \label{th:lip_hess}
  Let $y, z \in \Ys^{\eps}$ such that $\min_\ell y_\ell' \geq \mu$ and
  $\min_\ell z_\ell' \geq \mu$ for some constant $\mu > 0$, then
 \begin{align*}
   \big| \Es_\a''(y)[v, w] \big| \leq\,& C_{\rm H} \| v' \|_{\ell^2_\eps}
   \| w' \|_{\ell^2_\eps}, \qquad \text{and} \\
    \big| \{\Es_{\a}''(y) - \Es_{\a}''(z)\}[v, w] \big|
    \leq\,& C_{\rm Lip} \| y' - z' \|_{\ell_\eps^\infty} \| v'\|_{\ell^2_\eps}
    \|w'\|_{\ell^2_\eps}
    \qquad \forall v, w \in \Us,
  \end{align*}
  where $C_{\rm H} := M_2(\mu) + 4 M_2(2\mu)$ and $C_{\rm Lip} :=
  M_3(\mu) + 8 M_3(2\mu)$ and $M_j(t) := \max_{s \geq t}
  |\phi^{(j)}(s)|$.
\end{lemma}

\section{A Posteriori Error Estimates}
\label{Sec:A_Posteriori_Error}
\subsection{A posteriori error estimate for the solution}
We will {\em assume} the existence of an atomistic solution $y_\a$ in
a neighbourhood of $y_\qc$ (cf. \eqref{eq:w1inf_apriori_assmpt}), and
estimate the error $y_\a - I_\eps y_\qc$. It is in principle possible
to rigorously prove the existence of such a solution $y_\a$ in a
neighbourhoodo f $y_\qc$, following for example \cite{Ortner:2008a,
  Ortner:arXiv:0911.0671}, however, this would require substantial
additional technicalities. 


\begin{theorem}
  \label{th:grad_error}
  Let $y_{\qc}$ be a solution of the QC problem \eqref{LocalMinimumQC}
  with $\min_\ell (y_{\qc})'_\ell \geq r_* / 2$ and $A_*(y_{\qc}) >
  0$, where $A_*$ is defined in the statement of Lemma
  \ref{th:stab_lemma}. Suppose, further, that $y_\a$ is a solution of
  the atomistic model \eqref{Eq:LocalMinimaA} such that, for some
  $\tau > 0$,
  \begin{equation}
    \label{eq:w1inf_apriori_assmpt}
    \| y_\a' - y_{\qc}' \|_{L^\infty(\Omega)} \le \tau.
  \end{equation}
  If $\tau$ is sufficiently small, then we have the error estimate
  \begin{equation}
    \label{eq:main_errest}
    \| y_\a' - I_\eps y_{\qc}' \|_{L^2(\Omega)} \leq
    {\smfrac{2}{A_*(y_\qc)}} 
    \big( \eta(y_{\qc})+\eta^f + \eta^q \big),
  \end{equation}
  where $\eta$ is defined in \eqref{ResidualStoredEnergy} and $\eta^f$
  and $\eta^q$ are defined in \eqref{ResidualExternalForce}.
\end{theorem}

\begin{proof}
  Let $e := y_\a - I_\eps y_\qc$. We require that $\tau \leq \smfrac12
  \min y_\qc'$, then by the mean value theorem we know that there
  exists $\theta \in {\rm conv}\{y_\a, I_\eps y_{\qc}\}$, such that
  \begin{align*}
    E''_{\a}(\theta)[e,e] = & E'_{\a}(y_\a)[e] - E'_{\a}(I_\eps y_{\qc})[e]
    \nonumber\\
    =\,& E'_\qc(y_\qc)[I_h e] - E'_\a(I_\eps y_\qc)[e] \\
    = & \big( \Es'_{\qc}(y_\qc)[I_h e] -
    \Es'_{\a}(I_\eps y_{\qc})[e] \big) 
    -\big( \< f, I_h e\>_h -
    \<f, e\>_\eps \big).
  \end{align*}
  In this proof we write $\Es_\a(I_\eps y_\qc)$ to emphasize that we are
  comparing $y_\a'$ with $I_\eps y_\qc'$.

  Applying Theorem \ref{Theo:ResidualStoredEnergy} and Proposition
  \ref{Theo:RisidualExternalForce}, the two groups are respectively
  bounded by
  \begin{align*}
    \big| \Es'_{\qc}(y_\qc)[I_h e] -
    \Es'_{\a}(I_\eps y_{\qc})[e] \big| \leq\,& \eta(y_\qc) \| e'
    \|_{L^2}, \quad \text{and} \\
    \big| \< f, I_h e\>_h -
    \<f, e\>_\eps \big| \leq\,& (\eta^f + \eta^q) \| e' \|_{L^2},
  \end{align*}
  and hence we obtain
  \begin{equation}
    \label{ErrorEstimatShort}
    \Es''_{\a}(\theta)[e,e] \le \big(\eta(y_{\qc}) +
    \eta^f + \eta^q \big) \| e'\|_{L^2}.
  \end{equation}
  
  Next, we compute a lower bound on $\Es''_{\a} (\theta) [e,e]$. Using
  the Lipschitz estimate given in Lemma~\ref{th:lip_hess}, Proposition
  \ref{th:stab_lemma} together with our assumption that $y_\qc' \geq
  r_* / 2$, and the a priori bound \eqref{eq:w1inf_apriori_assmpt}, we
  have
  \begin{align*}
    \Es''_{\a} (\theta) [e,e] \geq\,& \Es''_\a(y_\qc)[e,e] - C_{\rm
      Lip} \| y_\a' - y_\qc' \|_{L^\infty} \| e' \|_{L^2}^2  \\
    \geq\,& \big(A_*(y_\qc) - C_{\rm Lip} \tau \big) \| e' \|_{L^2}^2.
  \end{align*}
  Hence, we if require that $\tau \leq A_*(y_\qc) / (2 C_{\rm Lip})$
  (since $C_{\rm Lip}$ is bounded as $\tau$ decreases this is
  satisfied for $\tau$ sufficiently small), then we arrive at
  \begin{displaymath}
    \smfrac12 A_*(y_\qc) \| e' \|_{L^2}^2 \leq \big(\eta^e(y_{\qc}) +
    \eta^f + \eta^q\big) \| e'\|_{L^2}.
  \end{displaymath}
  Dividing through by $\| e' \|_{L^2}$ yields the stated estimate.
\end{proof}

\subsection{A posteriori error estimate for the energy}
An important quantity of interest is the total energy of the system
being approximated. In this section, we derive an a posteriori
estimate for the energy difference $E_{\a} (y_\a) - E_{\qc}
(y_{\qc})$. To that end we decompose the energy difference into
\begin{equation}
  \label{eq:energy:decomp}
  \begin{split}
    \big|E_{\a}(y_\a) - E_{\qc} (y_{\qc})\big| =\,& \big|E_{\a}(y_\a) -
    E_{\a}(I_\eps y_{\qc})\big|
    + \big| \Es_{\a}(I_\eps y_{\qc}) - \Es_{\qc} (y_{\qc}) \big|\\
    & + \big| \<f, I_\eps y_\qc-Fx\>_\eps - \<f, y_\qc-Fx\>_h \big| 
  \end{split}
\end{equation}
and analyze each component separately.

For the first group on the right-hand side of
\eqref{eq:energy:decomp}, the result is standard:

\begin{lemma}
  \label{th:at_energy_est}
  Let $y, z \in \Ys^\eps$ such that $\min_\ell y_\ell' \geq \mu$ and
  $\min_\ell z_\ell' \geq \mu$ for some constant $\mu > 0$, and $y \in
  \argmin E_{\a}(\Ys^\eps)$, then
  \begin{equation}
    \big| E_{\a}(y) - E_{\a}(z) \big| \le \smfrac12 C_{\rm H} \| y' - z' \|^2_{L^2},
  \end{equation}
  where $C_{\rm H} = C_{\rm H}(\mu)$ was defined in Lemma
  \ref{th:lip_hess}.
\end{lemma}
\begin{proof}
  There exists $\theta \in {\rm conv}\{y, z\}$ such that
  \begin{displaymath}
    E_\a(z) = E_\a(y) + E_\a'(y)[z-y] + \smfrac12 E_\a''(\theta)[z-y,z-y].
  \end{displaymath}
  Since $E_\a'(y)[z-y] = 0$ and $E_\a'' = \Es_\a''$ applying Lemma
  \ref{th:lip_hess} yields the stated result.
\end{proof}

For the second group on the right-hand side of
\eqref{eq:energy:decomp}, the estimate is obtained from a
straightforward computation, using only the fact that the energy of a
bond lying entirely inside an element is exact in the QC energy. The
proof is omitted. Although the form of the error contributions $\mu_k$
looks complex at first glance, they are in fact straightforward to
compute.

\begin{lemma}
  \label{th:int_energy_est}
  For $y_h \in \Ys^h$ and $y'_h > 0$, we have
  \begin{align}
    \label{eq:int_energy_est}
    & \big| \Es_{\a}(y_h) - \Es_{\qc}(y_h) \big| \le \mu(y_h) := \sum_{k \in
      \Ks_\c} \mu_k, \qquad \text{where} \\
    \notag
    & \mu_k := \bigg|\sum_{\substack{b \in \Bs \\ x_k^h \in {\rm
          int}(b)}} 
    \bigg\{ \smfrac{|b \cap \Omega_\a|}{r_b}\Big[\phi(r_b D_b y_h) - \phi\big(r_b D_{b \cap \Omega_\a} y_h \big)\Big] 
            + \frac{1}{r_b}\int_{b \cap \Omega_\c^\per}\big[\phi(r_b D_b
            y_h) - \phi(r_b y_h')\big] \dx \bigg\} \bigg|.
  \end{align}
\end{lemma}

Finally, the third group on the right-hand side of
\eqref{eq:energy:decomp} can be estimated similarly as the external
residual, however since the ``test function'' $u_h = y_h - F x$ is now
known explicitly, some additional structure can be exploited. Note
that the error due to quadrature is again of higher order.

\begin{lemma}
  \label{th:ext_energy}
  Let $y_h \in \Ys^h$ and $u_h := y_h - Fx$, then 
  \begin{equation}
    \label{eq:err_ext_energy}
    \big| \<f, u_h\>_h - \< f, I_\eps u_h \>_\eps \big| \le \mu^f(y_h) +
    \mu^q(y_h),
  \end{equation}
  where $\mu^f(y_h)$ is the error due to external forces, and
  $\mu^q(y_h)$ the error due to quadrature. They are, respectively,
  defined as follows:
  \begin{align*}
    \mu^f(y_h) :=\,& \sum_{k \in \Ks_\c} \mu^f_k, \quad
    \text{where} \quad \mu^f_k := \smfrac{\eps}{2} \| f \|_{L^1(\omega_k)} \big|
    [u_h']_{x_k^h} \big|, \quad \text{and} \\
    \mu^q(y_h) :=\,& \sum_{k : T_k^h \subset \tilde\Omega_\c} \mu^q_k, \quad \text{where}
    \quad
    \mu_k^q := \smfrac{\eps^2}{4} \big\| (f I_\eps u_h)''
    \big\|_{L^1(T_k^h)} + \smfrac{h_k^2}{4} \big\| (f u_h)''
    \big\|_{L^1(T_k^h)}, 
  \end{align*}
  where the second derivative $(f I_\eps u_h)''$ is to be understood
  in the piecewise sense, $\omega_k = T_\ell^\eps$ if $(\ell-1)\eps <
  x_k^h < \ell\eps$ and $\omega_k = \emptyset$ if no such $(\ell-1)
  \in \Z$ exists, and $[u_h']_{x_k^h} := u_h'(x_k^h+) - u_h'(x_k^h-)$.
\end{lemma}

\begin{remark}
  1. It is straightforward to evaluate (possibly upper bounds of) the
  error contributions $\eta^f_k$ and $\eta^q_k$ with $O(K)$
  computational complexity. This is due to the fact that $u_h$ is
  piecewise linear on $T^h_k$, and $I_\eps u_h = u_h$ except in the
  neighourhoods $\omega_k$ of the element interfaces.

  2. For the purpose of mesh refinement, we will group the residual
  contribution of the elements touching the a/c interface, so that no
  error contributions is associated with the atomistic region, which
  cannot be further refined.
\end{remark}

\begin{proof}
  Similarly as in the external residual estimate we write the external
  energy difference as
  \begin{displaymath}
    \<f, I_\eps u_h\>_\eps -
    \<f, u_h\>_h = \int_\Omega \big[ I_\eps(f I_\eps u_h) - I_h(f u_h)\big] \dx.
  \end{displaymath}
  Using $I_\eps(f v) = I_h(f I_h v)$ in $\Omega \setminus
  \tilde{\Omega}_\c$, we decompose this difference into
\begin{align}
  \int_\Omega \big[ I_\eps(fI_\eps u_h) - I_h(f u_h) \big] \dx  
  =\,& \int_{\tilde\Omega_\c} \big[ I_\eps(f I_\eps u_h ) - I_h(f u_h)\big]
  \dx \nonumber \\
  \label{eq:int_eest:10}
  = &\int_{\tilde\Omega_\c} \big[ I_\eps(f I_\eps u_h) - f I_\eps u_h \big] \dx 
  + \int_{\tilde\Omega_\c} \big[ f I_\eps u_h - f u_h \big] \dx \\
  & \nonumber
  + \int_{\tilde\Omega_\c} \big[ f u_h - I_h(f u_h) \big] \dx
\end{align}

Unlike in the proof proof of Proposition
\ref{Theo:RisidualExternalForce}, where $v_h$ was unknown, we can use
our explicit knowledge of $u_h$, to estimate the first and third
groups on the right-hand side of \eqref{eq:int_eest:10} as follows:
\begin{equation}
  \label{eq:int_eest:15}
  \begin{split}
    \bigg|\int_{\tilde\Omega_\c} \big[ I_\eps(f I_\eps u_h) - f I_\eps u_h \big]
    \dx\bigg| \leq\,& \big\| \smfrac{\eps^2}{4} (f I_\eps u_h)'' \big\|_{L^1}, \qquad
    \text{and} \\
    \bigg|\int_{\tilde\Omega_\c} \big[ I_h(f u_h) - f u_h \big]
    \dx\bigg| \leq\,& \big\| \smfrac{h^2}{4} (f u_h)'' \big\|_{L^1},
  \end{split}
\end{equation}
where the second derivatives are understood in a piecewise sense. 

To estimate the second group on the right-hand side of
\eqref{eq:int_eest:10}, we note that $I_\eps u_h = u_h$ except near
element boundaries. Upon defining $\omega_k$ and $[u_h']_{x_k}$ as in
the statement of the result,
we have
\begin{align}
  \notag
  \int_{\tilde\Omega_\c} \big[ f I_\eps u_h - f u_h \big] \dx
  =\,& \sum_{\substack{k = 1, \dots, K \\ x_k^h \in \tilde\Omega_\c}} \int_{\omega_k}
  \big[ f I_\eps u_h - f u_h \big] \dx \\
  \label{eq:int_eest:20}
  \leq\,& \sum_{\substack{k = 1, \dots, K \\ x_k^h \in \tilde\Omega_\c}}
  \smfrac{\eps}{2} \| f \|_{L^1(\omega_k)} \big\| u_h' - I_\eps u_h'
  \big\|_{L^\infty(\omega_k)} \\
  \leq\,& \sum_{\substack{k = 1, \dots, K \\ x_k^h \in \tilde\Omega_\c}}
  \smfrac{\eps}{2} \| f \|_{L^1(\omega_k)} \big|[u_h']_{x_k}\big|.
\end{align}
It is now straightforward to rearrange the various error contributions
to obtain the stated result.
\end{proof}

Combining all the foregoing estimates yields an a posteriori
error estimate for the energy.

\begin{theorem}
  \label{th:energy_error_estimate}
  Suppose that the conditions of Theorem \ref{th:grad_error} are
  satisfied, then
  \begin{equation}
    \label{eq:energy_error}
    \big| E_\a(y_\a) - E_\qc(y_\qc) \big| \leq  \smfrac{4 C_{\rm H}}{A_*(y_\qc)^2} \big\{ \eta(y_\qc)^2 + (\eta^f)^2 +
    (\eta^q)^2 \big\} + \mu(y_\qc) +
    \mu^f(y_\qc) + \mu^q(y_\qc),
  \end{equation}
  where $\eta(y_\qc)$ is defined in \eqref{ResidualStoredEnergy},
  $\eta^f, \eta^q$ in \eqref{ResidualExternalForce}, $\mu(y_\qc)$ in
  \eqref{eq:int_energy_est} and $\mu^f(y_\qc), \mu^q(y_\qc)$
  in~\eqref{eq:err_ext_energy}.
\end{theorem}
\begin{proof}
  The second term on the right-hand side of \eqref{eq:energy:decomp}
  is estimated by $\mu(y_h)$, which gives rise to the second term on
  the right-hand side of \eqref{eq:energy_error} (cf. Lemma
  \ref{th:int_energy_est}). The third term on the right-hand side of
  \eqref{eq:energy:decomp} is estimated by $\mu^f(y_h)$ and
  $\mu^q(y_h)$, which gives rise to the third and fourth terms on the
  right-hand side of \eqref{eq:energy:decomp} (cf. Lemma
  \ref{th:ext_energy}).

  Finally, using Lemma \ref{th:at_energy_est}, the first term on the
  right-hand side of \eqref{eq:energy:decomp} can be bounded above by
  \begin{displaymath}
    \big| E_\a(y_\a) - E_\a(I_\eps y_\qc) \big| \leq \smfrac12 C_{\rm
      H} \| y_\a' - I_\eps y_\qc' \|_{L^2}^2.
  \end{displaymath}
  Employing Theorem \ref{th:grad_error} we obtain
  \begin{displaymath}
    \big| E_\a(y_\a) - E_\a(I_\eps y_\qc) \big| \leq
    \smfrac{8}{A_*(y_\qc)^2} \big( \eta(y_\qc)^2 + (\eta^f)^2 +
    (\eta^q)^2 \big). \qedhere
  \end{displaymath}
\end{proof}

\section{Numerical Experiments}
\label{Numerics}
In this section, we present numerical experiments to illustrate the
results of our analysis, and highlight further features of our a
posteriori error estimates. In particular, we will propose an adaptive
mesh refinement algorithm, and show numerically that it achieves an
optimal ``convergence rate'' in terms of the number of degrees of
freedom. (Strictly speaking, we cannot speak of convergence rates
since the algorithm will eventually fully resolve the problem.)

Throughout this section we fix $F = 1$, $N = 25000$, and let $\phi$ be
the Morse potential
\begin{equation*}
  \phi(r) = \exp(-2\alpha(r-1))-2\exp(-\alpha(r-1)),
\end{equation*}
with the parameter $\alpha = 5$.  We defined the external force $f$ to
be
\begin{align*}
  f(x) :=  \cases{
    - 0.4 \frac{x+1/2}{x},
    &\text{for $-1/2 \le x < 0$},\\[2mm]
    0.4 \frac{1/2-x}{x},
    &\text{for $0 < x \le 1/2$.}
  }
\end{align*}

Note that $f(x)$ behaves essentially like $|x|^{-1}$, which is a
typical decay rate for elastic fields generated by long-ranged defects
in 2D/3D. (By contrast local perturbations decay exponentially in our
1D model.) Thus, introducing this force allows us to study the
performance of our adaptive algorithm in a setting that includes some
of the features of 2D/3D problems. The constant $0.4$ is fairly
arbitrary. It was chosen sufficiently large to achieve a significant
deformation, but sufficiently small to ensure that there exists an
elastic stable equilibrium configuration.


We will analyze the relative errors of the deformation gradient and of the
energy defined, respectively, by
\begin{align}  
  \label{RelativeError}
  \frac{\| y_{\qc}'-y_\a'\|_{L^2(\Omega)}}{\|y_\a'- F
    \|_{L^2(\Omega)}} \qquad \text{and} \qquad 
  \frac{|E_{\a}(y_\a) - E_{\qc}(y_{\qc})|}{|E_{\a}(y_\a) - E_{\a}(Fx)|}.
\end{align}
In all computations, we compare the QC solutions against the
(computable) exact solutions.

\subsection{A priori mesh refinement}
We will compare the adaptive algorithm introduced in the next
sub-section against a quasi-optimal a priori mesh refinement scheme,
which exploits the known qualitative behaviour of the solution. For
simplicity we keep the following discussions informal.

We expect that, away from the defect, the exact solution will
essentially behave like $|u''| \lesssim |f|$. We choose to atomistic
region the be of the form $(- M\eps, M\eps)$. Closely
following the analysis of \cite[Sec. 7.1]{OrtShap2010b} to optimize
the mesh $\Ts^h$ based on these assumptions, we obtain that the
(quasi-)optimal mesh size in the continuum region is given,
approximately, by
\begin{displaymath}
  h^*(x) = 2 \eps \Big|\frac{f(M\eps)}{f(x)}\frac{|x|}{M\eps}\Big|^{\frac{2}{3}}.
\end{displaymath}
The following algorithm generates a mesh with this quasi-optimal mesh
size. We only state the construction for the right-hand half of the
domain. The mesh will be symmetric about the centre $x = 0$. The
factor $2 \eps$ ensures that the restriction (T4) is observed.

\begin{algorithm}[A priori mesh refinement]
\label{Algo:CHMesh}
\begin{enumerate}
\item Add the vertices $0, \eps, \dots, \eps M$ to the mesh.
\item Let $x_k^h$ be the right-most vertex already in the mesh. If
  $x_k^h + h^*(x_k^h) > N$, stop.
\item Otherwise, add the vertex $x_k^h + h^*(x_k^h)$ to the
  mesh. Continue at (2). 
\end{enumerate}
\end{algorithm}

\subsection{Adaptive algorithm}
In our adaptive computations, we begin with a mesh that resolves the
``defect'' (i.e., has atomistic resolution near $x = 0$) but is coarse
elsewhere. We then employ the algorithm stated below to locally refine
the mesh where required and thus improve the quality of the solution.

Before we state the algorithm, we first define the error indicators
according to which we drive the mesh refinement. Node-based error
error constributions are split between neighbouring elements. Error
contributions from the atomistic region are associated with the
neighbouring continuum elements.

The element error indicators for the gradient-driven algorithm are
given by (cf. \eqref{ResidualStoredEnergy} and
\eqref{ResidualExternalForce})
\begin{displaymath}
  (\rho_k^\nabla)^2 := \smfrac{4}{A_*(y_\qc)} \cdot  \cases{ 
    3 \eta_{k-1}^2 + \smfrac{3}{2} \eta_k^2 + (\eta_k^f)^2 + (\eta_{k-1}^f)^2, & \text{if
    } x_{k-1}^h = R_\a, \\
    \smfrac{3}{2} \eta_{k-1}^2 + 3 \eta_k^2 + (\eta_k^f)^2 +
    (\eta_{k+1}^f)^2, & \text{if } x_k^h = L_\a, \\
    \smfrac{3}{2} \eta_{k-1}^2 + \smfrac32 \eta_k^2 +
    (\eta_k^f)^2, & \text{otherwise}.
  }
\end{displaymath}
Note that we have ignored the quadrature contributions. We have
carefully justified this omission in Section \ref{sec:frc_est_sing}.

The element error indicators for the energy-driven algorithm are given
by (cf. Theorem~\ref{th:energy_error_estimate})
\begin{displaymath}
  \rho_k^E := \smfrac{4 C_{\rm H}}{A_*(y_\qc)^2} \, (\rho_k')^2 + \cases{
    \mu_{k-1} + \smfrac12 \mu_k +
    \mu_{k-1}^f + \smfrac12 \mu_k^f + \mu_{k-1}^q + \mu_k^q, &
    \text{if } x_k^h = R_\a, \\
    \smfrac12\mu_{k-1} + \mu_k +
    \smfrac12\mu_{k-1}^f + \mu_k^f + \mu_{k}^q + \mu_{k+1}^q, &
    \text{if } x_k^h = L_\a, \\
    \smfrac12\mu_{k-1} + \smfrac12\mu_k +
    \smfrac12\mu_{k-1}^f + \smfrac12\mu_k^f + \mu_{k}^q, &
    \text{otherwise}.\\
  }
\end{displaymath}
Since the energy error is formally second order, we did include the
quadrature contribution $\mu^q$ in this case.

In the following algorithm, let $\rho_k \in \{\rho_k^\nabla,
\rho_k^E\}$. Our algorithm is based on established ideas from the
adaptive finite element literature \cite{dorfler, mns2000}.

\begin{algorithm}[A posteriori mesh refinement]
\label{Algo:EDMesh}
\begin{enumerate}
\item Add the nodes $0, \pm \eps, \dots, \pm 3\eps$ to the mesh. Add
  the nodes $\{x_0^h = -1/2, x_K^h = 1/2\}$ to the mesh.
\item {\it Compute: } Compute the QC solution on the current mesh, compute the
  error indicators $\rho_k$. For $T_k^h \subset \Omega_\a$, set
  $\rho_k = 0$.
\item {\it Mark:} Choose a minimal subset $\mathcal{M} \subset \{1,
  \dots, K\}$ of indices such that
  \begin{equation}
    \sum_{k \in \mathcal{M}} \rho_k \le \frac12 \sum_{k = 1}^K \rho_k.
  \end{equation}
\item {\it Refine: } Bisect all elements $T_k^h$ with indices
  belonging to $\mathcal{M}$. \\
  If an element that needs to be refined is adjacent to the atomistic
  region, merge this element into the atomistic region and create a
  new atomistic to continuum interface.
\item If the resulting mesh reaches a prescibed maximal number of degrees of
  freedom, stop algorithm; otherwise, go to Step (2).
\end{enumerate}
\end{algorithm}

\subsection{Numerical Results}
We summarize the results of the computions with meshes generated by
the quasi-optimal refinement, and the adaptive algorithm with both
gradient- and energy-based error indicators.
\begin{figure}
  \includegraphics[width=11cm, bb = 55 160 650 650]{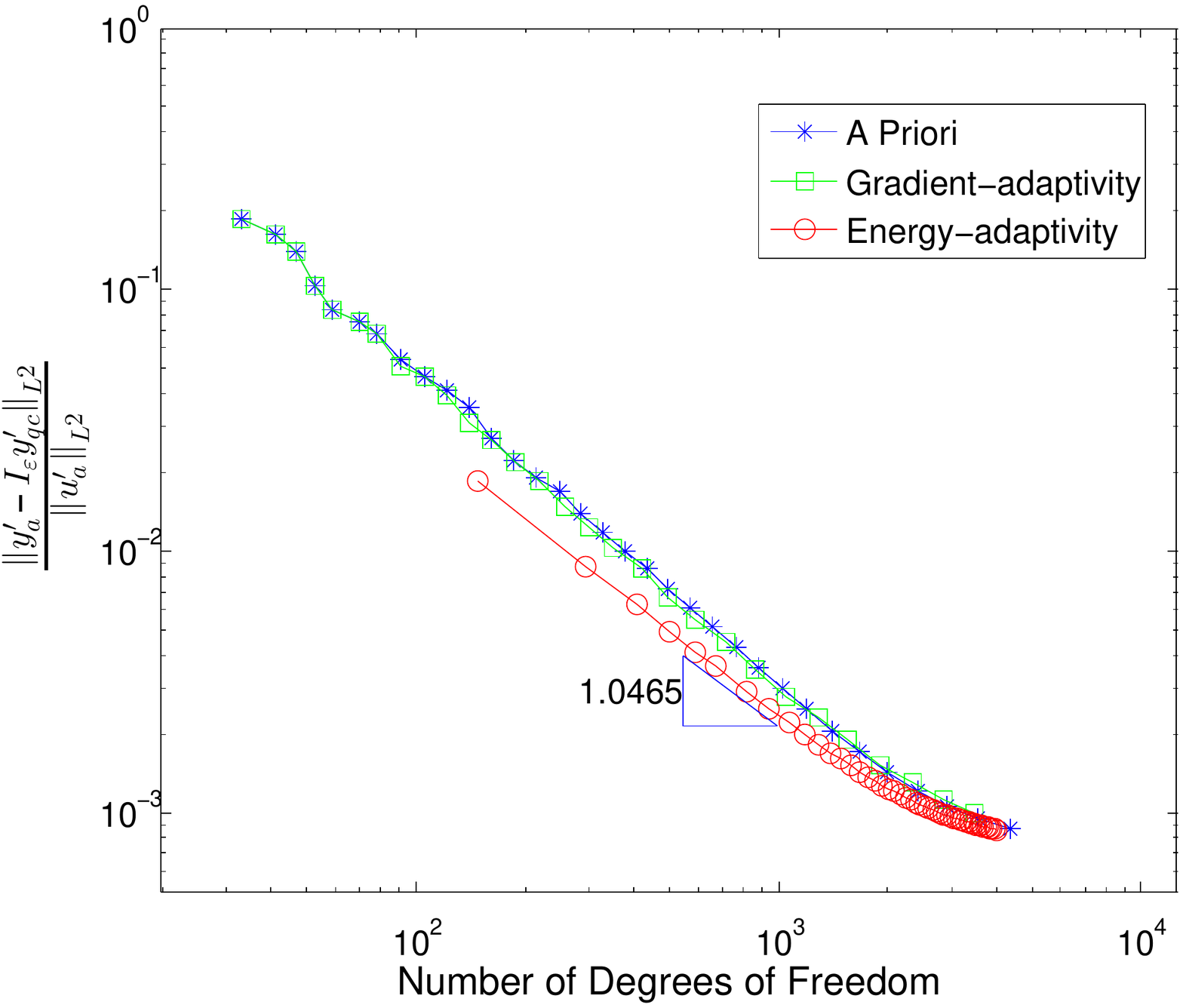}  
  \caption{Relative errors in the deformation gradient
    \eqref{RelativeError} plotted against the number of
    degrees of freedom for three types of mesh refinements.} 
  \label{Fig:Relative_Error_in_Gradient}
\end{figure}
\begin{figure}
  \includegraphics[width=11cm, bb = 55 160 650 650]{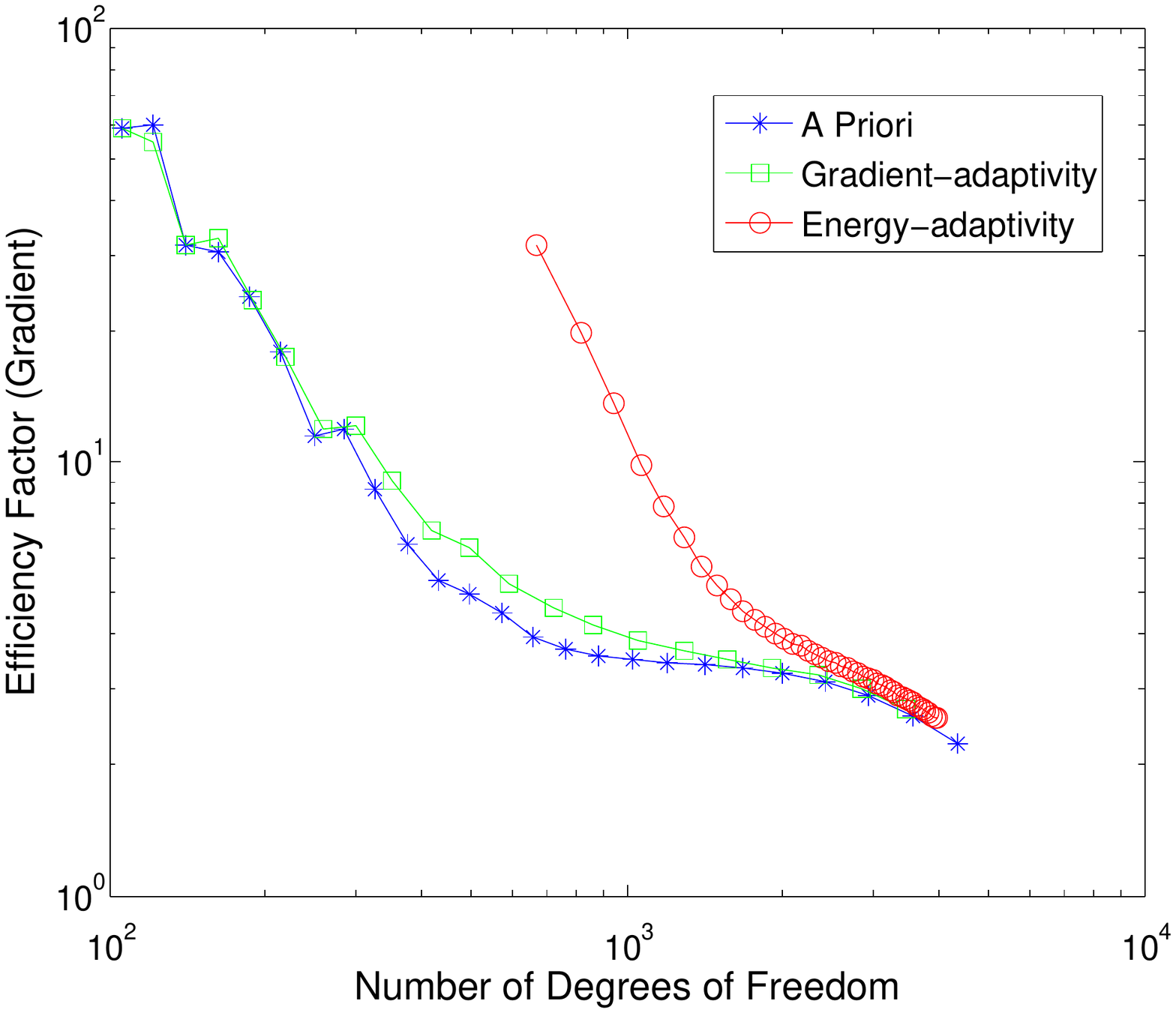}  
  \caption{Efficiency factors (gradient a posteriori estimate divided
    by actual error) plotted against the number of degrees of freedom
    for three types of mesh refinements.}
  \label{Fig:Efficiency_Factor_Gradient}
\end{figure}

\begin{enumerate}
\item In Figure \ref{Fig:Relative_Error_in_Gradient} we display the
  gradient errors for the three types of mesh generation algorithms:
  quasi-optimal a priori refinement, gradient-driven a posteriori
  refinement and energy-driven a posteriori refinement. The
  differences between the results produced by the three algorithms is
  negligable. We observe rates close to $(\#{\rm DoF})^{-1}$ for all
  three algorithms. The efficiency indicators (estimate divided by
  actual error) are displayed in Figure
  \ref{Fig:Efficiency_Factor_Gradient}. They indicate an initial
  overestimation of the error, but after a sufficient accuracy is
  reached they enter a moderate range.

\item In Figure \ref{Fig:Relative_Error_in_Energy} we show the energy
  errors for the three types of mesh generation algorithms. Once again
  the differences between the three algorithms is negligable and the
  convergence rate is, as expected, twice the rate as compared with
  the gradient errors. However, the efficiency factors plotted in
  Figure \ref{Fig:Efficiency_Factor_Energy} suggest that the constant
  prefactors in the estimators lead to a more substantial
  overestimation of the energy error.
\end{enumerate}

We can conclude that, at least in this model problem, both a
posteriori error indicators can be used to select meshes that are
quasi-optimal for both the deformation gradient and for the energy. In
order to obtain sharper estimates on actual errors (in particular the
energy error), alternative approaches such as the goal-oriented
approach \cite{Arndt2008b, Serge2007a} might be preferrable.

\begin{figure}
  \includegraphics[width=11cm, bb = 55 160 650 650]{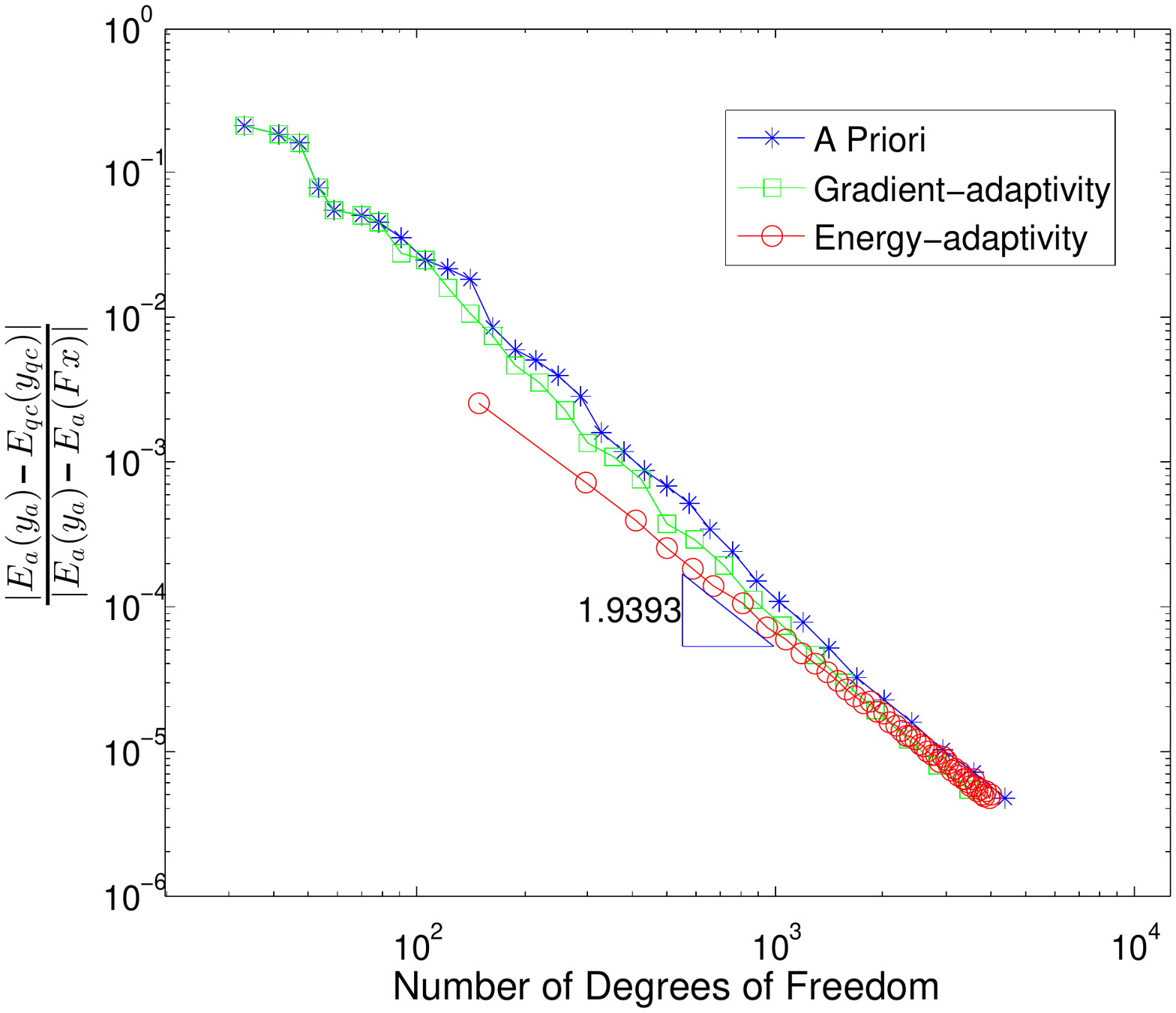}  
  \caption{Relative Error of the total energy
    \eqref{RelativeError} plotted against the number of degrees
    of freedom for three types of mesh refinements.}
  \label{Fig:Relative_Error_in_Energy}
\end{figure}

\begin{figure}
  \includegraphics[width=11cm, bb = 55 160 650 650]{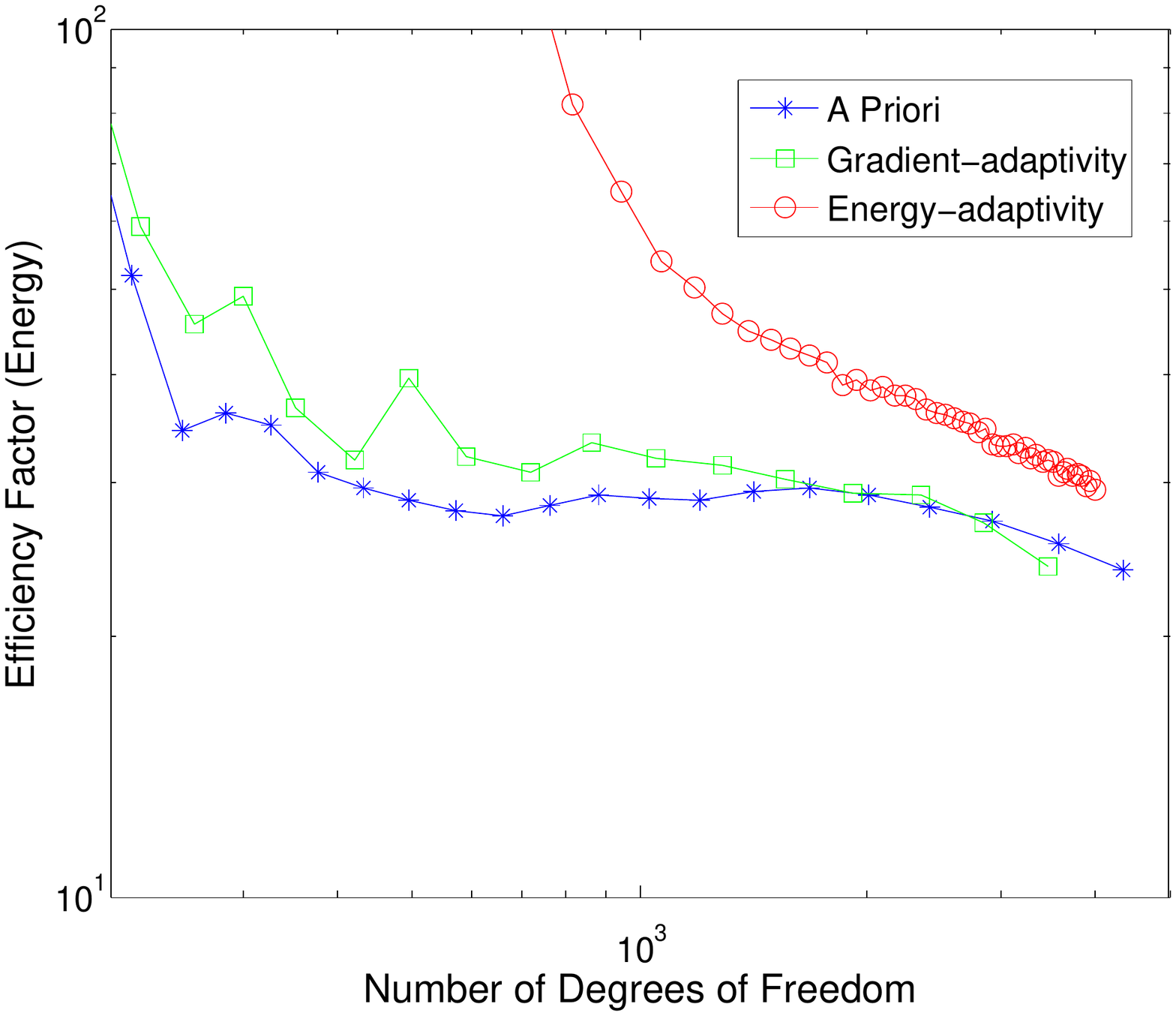}  
  \caption{Efficiency Factor of the energy error estimator (energy a
    posteriori error estimate divided by actual energy error).}
  \label{Fig:Efficiency_Factor_Energy}
\end{figure}

\section{Conclusion}
We derived a posteriori error estimators for the deformation gradient
and for the energy in a 1D atomistic chain, computed by an
atomistic-to-continuum coupling method.  Based on these estimators we
proposed two adaptive mesh refinement algorithms, which we compared to
quasi-optimal a priori mesh refinement. Our numerical experiments
indicate that the resulting meshes are again quasi-optimal.

While further work in the application relevant 2D/3D setting remains
to be done, we conclude that adaptive mesh refinement driven by
residual-based a posteriori error estimates can potentially lead to
highly efficient atomistic/continuum multiscale computations of
atomistic material defects.

\bibliographystyle{plain}
\bibliography{qc1}
\end{document}